\documentclass[12pt]{article}       
\usepackage{amssymb,verbatim}
\usepackage{amsmath,amsthm}
 \usepackage{graphicx}
\usepackage{enumitem}

\usepackage{lineno}
\usepackage{authblk} 

\usepackage[left=2.5cm,right=2.5cm,top=2.5cm,bottom=2.5cm]{geometry}
\makeatletter
\newtheorem*{rep@theorem}{\rep@title}
\newcommand{\newreptheorem}[2]{%
\newenvironment{rep#1}[1]{%
 \def\rep@title{#2 \ref{##1}}%
 \begin{rep@theorem}}%
 {\end{rep@theorem}}}
\makeatother

\newreptheorem{theorem}{Theorem}
\newreptheorem{corollary}{Corollary}
\newreptheorem{conjecture}{Conjecture}
\newtheorem{theorem}{Theorem}[section]

\newtheorem{lemma}[theorem]{Lemma}

\newtheorem{corollary}[theorem]{Corollary}
\newtheorem{conjecture}[theorem]{Conjecture}

\newtheorem*{remark}{Remark}
\newtheorem{example}{Example}
\newtheorem{problem}{Problem}

\newcommand{\h}{ \mathfrak{h}}
\newcommand{\lat}{\mathcal{L}}
\newcommand{\Qpl}{\mathcal{Q}}
\newcommand{\ZZ}{ \mathbb Z}
\newcommand{\NN}{ \mathbb N}
\newcommand{\QQ}{ \mathbb Q}

\newcommand{\RR}{ \mathbb R}
\newcommand{\inte}{\text{Int}}
\newcommand{\Int}{\text{Int}}

\title{Quasi-periodic tiling with multiplicity: a lattice enumeration approach}
\author{Swee Hong Chan\thanks{\noindent Department of Mathematics, 310 Malott Hall, Cornell University, \\
             Ithaca, NY 14853-4201 USA\\
Tel: (607) 255-4013   \ Fax: (607) 255-7149\\
              \texttt{sc2637@cornell.edu}           
}}

\date{\today}

\begin{document}
\maketitle
\begin{abstract}
The $k$-tiling problem for a convex polytope $P$ is  the problem of covering $\RR^d$ with  translates of  $P$ using a discrete multiset $\Lambda$ of translation vectors, such that
 every point in $\RR^d$ is covered exactly $k$ times, except possibly for the boundary of $P$ and its translates.
A classical result in the study of tiling problems is a theorem of McMullen \cite{McMullen}
 that a convex polytope $P$ that 1-tiles $\RR^d$ with a discrete multiset $\Lambda$ can, in fact, 1-tile $\RR^d$ with a lattice $\lat$.
A generalization of McMullen's theorem for $k$-tiling 
 was conjectured by Gravin, Robins, and Shiryaev \cite{Gravin-Robins-Shiryaev}, which states    that if  $P$ $k$-tiles $\RR^d$ with a discrete multiset $\Lambda$, then $P$  $m-$tiles $\RR^d$ with a lattice $\lat$ for some $m$.  
In this paper, we consider the case when $P$ $k$-tiles $\RR^d$ with a discrete multiset $\Lambda$ such that every element of $\Lambda$  is contained in a quasi-periodic set $\Qpl$ (i.e. a finite union of translated lattices).
 This is motivated by the  result  of Gravin, Kolountzakis, Robins, and Shiryaev \cite{Kolountzakis,GKRS}, showing  that for $d \in \{2,3\}$, if a polytope $P$ $k$-tiles $\RR^d$ with a discrete multiset $\Lambda$, then $P$  $m$-tiles $\RR^d$ with a  quasi-periodic set $\Qpl$ for some $m$.
 Here we show for all values of $d$ that if a polytope $P$ $k$-tiles  $\RR^d$ with  a discrete multiset $\Lambda$ that is contained in a quasi-periodic set $\Qpl$ that satisfies a mild hypothesis, then $P$ $m$-tiles $\RR^d$ with a lattice $\lat$ for some $m$.
This strengthens the results of Gravin, Kolountzakis, Robins, and Shiryaev, and  is a step in the direction of proving the conjecture of Gravin et al. \cite{Gravin-Robins-Shiryaev}.

\smallskip
\noindent \textbf{Subclass:} {52C22 }

\end{abstract}
\section{Introduction}\label{Introduction}
The multiple tiling problem can be described as follows:  cover every point in $\RR^d$ exactly $k$ times by translates of  a convex polytope using a discrete multiset $\Lambda$ (also known as a \textbf{tiling set}) of translation vectors. 
However, in the process of trying to cover every point in $\RR^d$,  we may be forced to cover points in the boundary of its  translates for  more than $k$ times. 
To avoid this technicality, we say that $P$ $k$-tiles $\RR^d$ with $\Lambda$ if every point that does not belong  to the boundary of any translate of $P$ is covered exactly $k$ times.
We call a polytope $P$ that satisfies the condition above  a $\mathbf{k}$\textbf{-tiler}. 
In the special case when $k$ is equal to $1$, the multiple tiling problem becomes what is traditionally known as the  translational tiling problem. For more details on the translational tiling problem, the reader is referred to   \cite{Alexandrov,Gruber,KM10} for  a nice overview of  the topic.


The translational tiling problem is a classical topic in discrete geometry with several beautiful structural results.
 In 1897, Minkowski \cite{Minkowski} gave a necessary condition for a polytope to be a $1$-tiler. He proved that  if a convex polytope $P$ $1$-tiles $\RR^d$, then $P$  must be centrally symmetric and all facets of $P$ must be centrally symmetric.  
 It was not until 50 years later that Venkov  \cite{Venkov} found a necessary and sufficient condition for a polytope $P$ to be a 1-tiler.
 He showed that   $P$ 1-tiles $\RR^d$ if and only if $P$ is centrally symmetric, all facets of $P$ are centrally symmetric, and  each belt of $P$ contains four or six facets.
  The same result was later rediscovered independently by McMullen \cite{McMullen} in 1980.

In contrast to the situation for $1$-tilers, there are still a lot of unsolved  problems on the structure of $k$-tilers. This is partly because $k$-tilers have a much richer structure compared to $1$-tilers. For example, in two dimensions there are only two types of convex polytopes that  $1$-tile $\RR^2$, namely centrally symmetric parallelograms and centrally symmetric hexagons. In contrast, all centrally symmetric integer polygons are $k$-tilers in  $\RR^2$ \cite{Gravin-Robins-Shiryaev}.  With that being said, there are several important results for multiple tiling that mirror the results for  $1$-tiling.    In 1994, Bolle \cite{Bolle} gave a necessary and sufficient condition for  a polytope to $k$-tile $\RR^2$ with a lattice, and in 2012, Gravin, Robins, and Shiryaev \cite{Gravin-Robins-Shiryaev} proved that a $k$-tiler in $\RR^d$ must be centrally symmetric and all its facets must be centrally symmetric, providing a multiple tiling  analogue for Minkowski's condition. 
%

The main motivation for this paper comes from  a classical result of McMullen \cite{McMullen}  that  if a convex polytope $P$  $1$-tiles  $\RR^d$ with a discrete multiset $\Lambda$, then $P$ can, in fact, $1$-tile $\RR^d$ with a lattice $\lat$.
A  generalization of McMullen's theorem for $k$-tiling  was conjectured by Gravin, Robins, and Shiryaev \cite{Gravin-Robins-Shiryaev}, and is stated below: 
 \begin{conjecture}\label{conjecture last boss}\textnormal{\cite[Conjecture 7.3]{Gravin-Robins-Shiryaev}}
If  a convex polytope $P$ $k$-tiles $\RR^d$ with a discrete multiset $\Lambda$, then $P$ $m$-tiles $\RR^d$ with a lattice $\lat$ for some $m$ (not necessarily equal to $k$).  
\end{conjecture}
There has been some recent progress on  Conjecture \ref{conjecture last boss} in lower dimensions.  
We say that a set $\Qpl$ is a \textbf{quasi-periodic set}  if $\Qpl$ is a finite union of translated 
lattices, not necessarily of the same lattice.
Kolountzakis \cite{Kolountzakis}  showed that if $P$ $k$-tiles  $\RR^2$ with a discrete multiset $\Lambda$, then $P$ can $m$-tile $\RR^2$ with a quasi-periodic set for some $m$. Kolountzakis' result was later extended by Shiryaev \cite{Shiryaev}, who proved  Conjecture \ref{conjecture last boss} when $d$ is equal to 2.
For the case when $d$ is equal to 3, Gravin, Kolountzakis, Robins, and Shiryaev \cite{GKRS} showed  that every $k$-tiler $P$ in $\RR^3$ can $m$-tile $\RR^3$ with a quasi-periodic set for some $m$.
Motivated by the results of \cite{Kolountzakis} and \cite{GKRS}, 
 we focus here on studying the 
\textbf{quasi-periodic tiling} problem for general dimensions,
which is the multiple tiling problem with an additional assumption that every element of the tiling multiset $\Lambda$ is contained in a {quasi-periodic set} (note that the multiset $\Lambda$  is not necessarily a quasi-periodic set).

We approach the quasi-periodic tiling problem by  studying  an equivalent lattice-point enumeration problem, which will be described in Section \ref{section lattice point enumeration}.  
This approach allows us to employ several  tools from  lattice-point enumeration that are otherwise 
 not available for  general multiple tiling problems. 
For more details regarding  discrete-point enumeration of polytopes, the reader is referred  to  the work of Beck and Robins \cite{Beck-Sinai} and  Barvinok \cite{Sasha-Barvinok}.

Throughout this paper, we will use two different notions for \textbf{general positions}, one for vectors and one for lattices. Let $P$ be a  a fixed convex polytope and let $\partial P$ denote the boundary of $P$.
 We say that a vector $v \in \RR^d$ is in  \textbf{general position} with respect to a discrete multiset $\Lambda$ if $v$ is not contained in $\partial P+\Lambda$, the union of translates of boundary of $P$ by $\Lambda$. The second notion is defined for lattices in $\RR^d$. Let $\Qpl$ be a union of $n$ translated lattices $\lat_1$, $\lat_2,\ldots, 
\lat_n$. 
We say that  $\lat_i$  is in  \textbf{ general position} with respect to $\Qpl$ if the set $\RR^d \setminus H_i$ is path-connected, where $H_i$ is defined as:
\begin{equation}\label{technical condition}
H_i:= \bigcup_{j=1, j\neq i}^{n} (\partial P+\lat_i) \cap (\partial P+\lat_j).
\end{equation}
We will refer to 
the hypothesis that $\lat_i$ is in  general position with respect to $\Qpl$  as Hypothesis \ref{technical condition}. 


The first result in this paper is  that if   Hypothesis \ref{technical condition} holds, then   Conjecture \ref{conjecture last boss} is true. 
\begin{theorem}\label{main theorem 1}
Let $P$ be a convex polytope that $k$-tiles $\RR^d$ with a discrete multiset $\Lambda$, and suppose that every element  of $\Lambda$ is  contained in a quasi-periodic set $\Qpl$. If a lattice $\lat$ in $\Qpl$ is in  general position with respect to $\Qpl$ and $\lat \cap \Lambda$ is non-empty,  then $P$ $m$-tiles $\RR^d$ with $\lat$ for some $m$. 
\end{theorem}

 If the quasi-periodic set  $\Qpl$ in Theorem \ref{main theorem 1} is also a lattice, then  the lattice $\Qpl$ is in  general position with respect to $\Qpl$ by definition, and Theorem \ref{main theorem 1} gives us the following corollary.
 
\begin{corollary}\label{corollary one lattice}
Let  $P$ be a convex polytope that $k$-tiles $\RR^d$ with a discrete multiset $\Lambda$.
 If every element of $\Lambda$ is  contained in a lattice $\lat$, then $P$  $m$-tiles $\RR^d$ with $\lat$ for some $m$. \qed 
\end{corollary}
 
 We note that  Theorem \ref{main theorem 1} will fail to hold if the hypothesis that  $\lat$ is in  general position with respect to $\Qpl$ is omitted.
 In Example \ref{example rectangle},  we  present a a polytope  $P$ that $k$-tiles $\RR^d$ with a quasi-periodic set $\Qpl$ that contains a lattice $\lat$, and yet $\lat$ is not a tiling set of $P$.
  
The fact that Hypothesis \ref{technical condition} cannot be omitted from Theorem \ref{main theorem 1} naturally leads us to consider  the case where every lattice in $\Qpl$ is not in  general position with respect to $\Qpl$.
In Section \ref{section for main theorem 2} we address this scenario
 for the case where $\Qpl$ is a union of two translated copies of a lattice, and we show that 
  Conjecture \ref{conjecture last boss} holds in this case. 
\begin{theorem}\label{main theorem 2}
Let $P$ be a convex polytope that $k$-tiles $\RR^d$ with a discrete multiset $\Lambda$. 
If  every element of $\Lambda$ is contained in  a union of two translated copies of one single lattice, 
then there is a  lattice $\lat$ in $\RR^d$ such that $P$ $m$-tiles $\RR^d$ with $\lat$ for some $m$. 
\end{theorem}


The paper is organized as follows. In Section \ref{definition} we introduce definitions and notations used in this paper. 
We  use Section \ref{section lattice point enumeration} to establish the connection between the $k$-tiling problem and a lattice-point enumeration problem.  
Section \ref{section main theorem 1} is devoted to the proof of Theorem \ref{main theorem 1}, and Section \ref{section for main theorem 2} is devoted to the proof of Theorem \ref{main theorem 2}. 
Finally, in  Section \ref{conjectures} we  discuss possible future research that can be done to prove  Conjecture \ref{conjecture last boss}.

\section{Definitions and preliminaries}\label{definition}
Throughout this paper, we use $P$ to denote a convex polytope in $\RR^d$,  $\Int(P)$ to denote the interior of $P$, and $\partial P$ to denote the boundary of $P$ (the closure of $P$ minus the interior of P).
We note that   there is no loss in generality in assuming that $P$ is a convex polytope, because 
every  convex body that $k$-tiles $\RR^d$ is necessarily a polytope \cite{McMullen}.

We use  $\Lambda$  to denote a discrete multiset of vectors in $\RR^d$, $\lat$ to denote a lattice in $\RR^d$, and $\Qpl$ to  denote a \textbf{quasi-periodic set}, which is a finite union of translated lattices, not necessarily of the same lattice.  
We  use $\#(A)$ to denote the cardinality of a finite multiset $A$ (counted with multiplicities).
The intersection of a multiset $A$ and a set $S$, denoted by $A \cap S$, is  the multiset that contains all elements $a$ in $A$ that are also contained in $S$.
The multiplicity of an element $a$ in $A \cap S$ is equal to the multiplicity of $a$ in $A$.
 The complement of a set $S$ with respect to a multiset $A$, denoted by $A\setminus S$, is the set $A \cap S^c$.

A convex polytope $P$ is said to \textbf{ k-tile } $\RR^d$ ($k$ being a positive natural number) with  a discrete multiset $\Lambda$ of vectors in $\RR^d$ if  
\begin{equation}\label{equation k-tiling}\sum_{\lambda \in \Lambda} \boldsymbol{1}_{P+\lambda}(v)=k,\end{equation}
 for all $v \notin \partial P+\Lambda$, where $\boldsymbol{1}_X$ is the indicator function of the set $X$. 
 
%

Throughout this paper, we assume that  $\h$ is a  fixed vector  in $\RR^d$ such that every line with direction vector $\h$  meets $\partial P$ 
at finitely many points. 
The \textbf{half-open} counterpart $P^{\h}$ of a convex polytope $P$ is the subset of the closure of $P$ that contains all points $v \in \RR^d$ 
which satisfies the property that
  for a sufficiently small $\epsilon_v>0$, the ray $r_{\epsilon_v}:=\{v+ c \h \ | \ 0 < c < \epsilon_v \} $ is contained in  $\inte(P)$. 
  Note that $P^{\h}$ consists of $\Int(P)$ and a part of $\partial P$.
   In the particular case when $P$ is a cube, 
the polytope $P^{\h}$ is 
  the half-open cube  defined in \cite{Stanley}.
  
   For a a discrete multiset $\Lambda$ and a convex polytope $P$ in $\RR^d$, the \textbf{$\Lambda$-point enumerator} of $P$ is the integer $\#(\Lambda \cap P)$, which is the number of points of $\Lambda$ (counted with multiplicities) contained in $P$. When $\Lambda$ is a lattice, we refer to $\#(\Lambda \cap P)$ as a lattice-point enumerator. 
   We  define two integer-valued functions, $L_{\Lambda}$ and $L_{\Lambda}^{\h}$,  on every point $v$ in $\RR^d$ as follows: 
   \[L_{\Lambda}(v):= \# (\Lambda \cap  \{-1 \cdot P+ v \}),   \ L_{\Lambda}^{\h}(v):=\# (\Lambda \cap  \{-1 \cdot P^{\h}+ v \}), \] i.e. $L_{\Lambda}(v)$ is the number of points of $\Lambda$ contained in the translate of $-1 \cdot P$ by $v$, and  $L_{\Lambda}^{\h}(v)$ is the number of points of $\Lambda$ contained in the translate of  $-1 \cdot P^{\h}$ by $v$. If the intended multiset $\Lambda$ is evident from the context,
  we will use $L$ and $L^{\h}$ as a shorthand for $L_{\Lambda}$ and  $L_{\Lambda}^{\h}$, respectively. 
These two functions  play an important role in relating the multiple tiling problem to the lattice-point enumeration problem, which is discussed in Section \ref{section lattice point enumeration}.

%
%
%
\section{Lattice-point enumeration of polytopes}\label{section lattice point enumeration}

In this section, we  present a lattice-point enumeration problem that is equivalent to $k$-tiling problem.  
This equivalence  was first shown in \cite{Gravin-Robins-Shiryaev}, where  it was employed to show that all rational $k$-tilers can $m$-tile with a lattice for some $m$. 
However,  we replace the polytope $P$ by its half-open counterpart $P^{\h}$ in the statement of the equivalence. 
This is done so that we can drop the technical condition in the equivalence concerning vectors in general position. 


\begin{lemma}\label{half open}
A $d$-dimensional convex polytope $P$ $k$-tiles $\RR^d$ with a discrete multiset $\Lambda$  if and only if its half-open counterpart $P^{\h}$ $k-$tiles $\RR^d$ with $\Lambda$. Moreover, if $P^{\h}$ $k$-tiles $\RR^d$ with $\Lambda$, then
\begin{equation}\label{equation half open}
\sum_{\lambda \in \Lambda} \boldsymbol{1}_{P^{\h}+\lambda}(v)=k,
\end{equation}
for all $v $ in $\RR^d$.
\end{lemma} 

\begin{proof}
Note that $P$ and $P^{\h}$ have the same interior, which  implies that
\[\sum_{\lambda \in \Lambda}\boldsymbol{1}_{P^{\h}+\lambda}(v)=\sum_{\lambda \in \Lambda} \boldsymbol{1}_{P+\lambda}(v), \]
for all $v \notin \partial P+\Lambda$. Therefore, $P$ $k$-tiles $\RR^d$ with $\Lambda$ if and only if $P^{\h}$ $k$-tiles $\RR^d$ with $\Lambda$. 

To prove the second part of the claim, let $v$ be an arbitrary point in $\RR^d$. By our assumption on $\h$, the ray $r_{\epsilon_v}= \{v+ c\h \ | \  0< c < \epsilon_v \}$   intersects $\partial P+\Lambda$ at finitely many points.
 Hence for a sufficiently small $\epsilon_v>0$, the ray $r_{\epsilon_v}$ does not intersect $\partial P+\Lambda$. 
Because $P^{\h}$(and hence $P$) $k$-tiles $\RR^d$ with $\Lambda$, this implies that there are exactly $k$ vectors $\lambda_1, \ldots, \lambda_k$ in $\Lambda$ such that $r_{\epsilon_v}$ is contained in the interior of $P+\lambda_i$ for all $i$. By the definition of half-open polytopes in Section \ref{definition}, this means that $v$ is contained in $P^{\h}+\lambda_i$ for all $i$, and hence we have:
\begin{equation*}\sum_{\lambda \in \Lambda} \boldsymbol{1}_{P^{\h}+\lambda}(v)=k,\end{equation*}
for all $v \in \RR^d$.
\qed \end{proof}

 The lemma below was shown in \cite{Gravin-Robins-Shiryaev} for  the case when $v \in \RR^d$ is in  general position with respect to $\Lambda$; 
  our proof is virtually identical, with a minor adjustment for the case when $P$ is replaced by $P^{\h}$.

\begin{lemma}\label{lemma GRS}\textnormal{(c.f. \cite[Lemma 3.1]{Gravin-Robins-Shiryaev})} A convex polytope $P$ $k$-tiles $\RR^d$ with a  discrete multiset $\Lambda$  if and only $L_{\Lambda}( v)$ is equal to $k$
 for every $v \in \RR^d$ that is in  general position with respect to $\Lambda$. 
 Moreover,  $P^{\h}$ $k$-tiles $\RR^d$ with $\Lambda$ if and only $L_{\Lambda}^{\h}( v)$ is equal to $k$ for every $v$ in $\RR^d$. 
\end{lemma}

\begin{proof}
For every $v$ in $\RR^d$, we can write
\[\sum_{\lambda \in \Lambda} \boldsymbol{1}_{P^\h+\lambda}(v)= \sum_{\lambda \in \Lambda} \boldsymbol{1}_{-1 \cdot P^\h+v}(\lambda)= \# (\Lambda \cap \{-1 \cdot P^\h + v\})=L_{\Lambda}( v).\]
By Equation \ref{equation half open} in   Lemma \ref{half open}, this implies that $P^{\h}$  $k$-tiles $\RR^d$ if and only if $L_{\Lambda}^{\h}(v)$  is equal to $k$ for every $v$ in $\RR^d$. 
By a similar argument, 
 Equation \ref{equation k-tiling} in the definition of $k$-tilers implies that  $P$ $k$-tiles $\RR^d$ if and only if $L_{\Lambda}( v)
$ is equal to $k$ for every $v \in \RR^d$ that is  in  general position with respect to $\Lambda$.  
\qed \end{proof}


Here we list two  properties of the function $L_\Lambda$ that will be used throughout this paper: 
\begin{enumerate}[leftmargin=*]
 \item \label{property periodic} If $\lat$ is a lattice in $\RR^d$, then the function $L_{\lat}$ is a periodic function of $\lat$ (i.e.  $L_{\lat}(v+\lambda)=L_{\lat}(v) $ for every $v$ in $\RR^d$ and every $\lambda$ in $\lat$).  This is because a lattice-point enumerator is invariant under translation by elements contained in the lattice. 
 \item \label{property general position} The function $L_{\Lambda}$ is a constant function in a sufficiently small neighborhood $B_v$ of $v$ in $\RR^d$ if $v$ is in  general position with respect to $\Lambda$. This is because if $v \notin \partial P+\Lambda$, then $-1 \cdot P +v$ does not contain any points from $\Lambda$ in its boundary. Hence moving $v$ in any sufficiently small direction will not change the value of $L_{\Lambda}(v)$. 
\end{enumerate}
It can be easily checked that the two properties above also hold for the function $L_{\Lambda}^{\h}$.

%
%
%

\section{Proof of Theorem  \ref{main theorem 1} } \label{section main theorem 1}
We  start this section by proving a functional equation involving the function $L_{\Lambda}$. 
The proof borrows several ideas from asymptotic analysis of infinite sums in $\RR^d$.
\begin{lemma}\label{asymptotic method}
Let $P$ be a convex polytope that $k$-tiles $\RR^d$ with a discrete multiset  $\Lambda$.  
Let $\lat$ be a lattice and let $a_1,\ldots, a_n$ be vectors in $\RR^d$.
If    every element of $\Lambda$ is contained in the 
finite union  $\bigcup_{i=1}^{n} a_i+\lat$, then
there are non-negative real numbers   $g_1, \ldots, g_n$ such that
\begin{equation}\label{equation in props asymptotic}\sum_{i=1}^{n}  g_i \cdot L_{\lat}^{\h}(v-a_i)=k,
\end{equation}
for all $v \in \RR^d$.
\end{lemma}

\begin{proof}
Without loss of generality, we can assume that $\lat=\ZZ^d$. 
%
For a real vector $w=(w_1, \ldots, w_d)$ in $\RR^d$, we use  $\RR^d_{\geq w}$ to denote the  set $\{(w_1', \ldots, w_d'): w_i' \geq w_i \text{ for } i \in \{1,\ldots,d\} \}$.
We use  $ \Lambda_{\geq 0}$ to denote  the multiset $\Lambda \cap \RR^d_{\geq0}$.

Let $v \in \RR^d$ be an arbitrary vector.
Because $P^{\h}$ $k$-tiles $\RR^d$ with $\Lambda$, this  implies   that there is   a vector ${\alpha(v)} \in \RR^d$ such that every point in $\RR^{d}_{\geq {\alpha(v)}}$  is covered exactly $k$ times by $P^{\h}+v+\Lambda_{\geq 0}$. Also notice that because $\Lambda_{\geq 0}$ is in the positive orthant, there is a vector $\beta(v)$ in $\RR^d$ such that $P^{\h}+v+\Lambda_{\geq 0}$ is contained in  $\RR^d_{\geq \beta(v)}$. Without loss of generality, we can assume that both ${\alpha(v)}$ and $\beta(v)$ are integer vectors. 

Let $\Gamma(v)$ be the multiset $ ( P^{\h}+v+\Lambda_{\geq 0}) \setminus \RR^d_{\geq {\alpha(v)}} $. 
Because $ P^{\h}+v+\Lambda_{\geq 0}$ covers every point in $\RR^d_{\geq {\alpha(v)}}$ exactly $k$ times, we  have the following equality:
\begin{align}\label{equation start}
\sum_{\substack{x \in P^{\h}+v+\Lambda_{\geq 0}, \\ x \in \ZZ^d}} z^x &= \sum_{x \in \ZZ^d_{\geq {\alpha(v)}} } kz^x + \sum_{x \in \Gamma(v) \cap \ZZ^d}z^x,
\end{align}
where $z^x$ is the multivariable polynomial $z_1^{x_1}\ldots z_d^{x_d}$ and  $x$ is an integer vector $(x_1, \ldots, x_d)$.
  We  assume $|z_j|<1$ so  that all the sums in Equation \ref{equation start}  converge to a well-defined value. 

  We define the multisets  $\Lambda_i$ for $i\in\{1,\ldots, n\}$  recursively by   $\Lambda_i:= ( \Lambda_{\geq 0} \cap \{a_i +\ZZ^d\} )\setminus \bigcup_{j=1}^{i-1} \Lambda_j$ for $ i \in \{1,\ldots, n\}$. 
  Note that  $\boldsymbol{1}_{\Lambda_{\geq 0}}= \sum_{i=1}^{n} \boldsymbol{1}_{\Lambda_i}$, and  every element of $\Lambda_i-a_i$ is contained in $\ZZ^d$. 
  With this notation, Equation \ref{equation start} now becomes 
  \begin{align}
   \sum_{x \in \ZZ^d_{\geq {\alpha(v)}}  } kz^x + \sum_{x \in \Gamma(v) \cap \ZZ^d}z^x  
  = \sum_{i=1}^{n} \sum_{\substack{x \in P^{\h}+v+\Lambda_i,\\  x \in \ZZ^d}} z^x  
 & =\sum_{i=1}^{n} \left(\sum_{\substack{x \in P^{\h}+v+a_i, \\  x \in \ZZ^d}} z^x  \cdot \sum_{y \in \Lambda_i -a_i} z^y   \right) \label{equation intermediate}\\
&= \sum_{i=1}^{n} \# (\ZZ^d \cap \{ P^{\h}+a_i+v \} ) \cdot  \sum_{y \in \Lambda_i -a_i} z^y. \label{equation important}
\end{align}

Let $s_i(z):=\sum_{y \in \Lambda_i -a_i} z^y \cdot \prod_{j=1}^d (1-z_j)$.
By 
 multiplying 
   Equation \ref{equation important} with $\prod_{j=1}^d (1-z_j) $, we obtain that:
  \begin{align}\label{equation no limit}
   \sum_{x \in \ZZ^d_{\geq {\alpha(v)}}  } kz^x \cdot \prod_{j=1}^d (1-z_j) + \sum_{x \in \Gamma(v) \cap \ZZ^d}z^x \cdot  \prod_{j=1}^d (1-z_j)=\sum_{i=1}^{n} \# (\ZZ^d \cap \{ P^{\h}+a_i+v \} ) \cdot s_i(z)
   \end{align}
We now show that the left side of Equation \ref{equation no limit} converges to $k$ as $z$ converges to $(1,\ldots,1)$ from below.

First note that every element in $\Gamma(v)$   is contained in $\RR^d_{\geq \beta(v)}  \setminus  \RR^d_{\geq {\alpha(v)}}  $, and every element of $\Gamma(v)$ has multiplicity at most $k$. 
Also note that $\RR^d_{\geq \beta(v)}  \setminus  \RR^d_{\geq {\alpha(v)}}  $
 is contained in  the set $\bigcup_{i=1}^d R_i(v)$, where $R_i(v)$ is the set
 \[R_i(v):=\{(a_1, \ldots, a_d) \ : \ {\beta(v)}_i \leq a_i < {\alpha(v)}_i  \text{ and }  {\beta(v)}_j\leq  a_j \text{ for all } j \neq i , \ 1\leq j \leq d \ \}.\]

Because $|z_i|<1$  for all $i$,  we  have the following closed-form expressions:
\begin{align}\label{equation stripes}
\sum_{x \in \ZZ^d_{\geq {\alpha(v)}}} z^x= \prod_{j=1}^d \frac{z_j^{{\alpha(v)}_j}}{1-z_j} ; \ 
\sum_{x \in R_i(v) \cap \ZZ^d}z^x =(1-z_i^{{\alpha(v)}_i-\beta(v)_i}) \cdot \prod_{j=1}^{d}\frac{z_j^{{\beta(v)}_j}}{1-z_j}.
\end{align}

  Because $\Gamma(v)$ is contained in $\bigcup_{i=1}^{d} R_i(v)$ and each element in $\Gamma(v)$ has multiplicity at most  $k$, we  conclude that:
  \begin{align}\label{equation gamma}
  \lim_{z \to (1,\ldots,1)^-} \sum_{x \in \Gamma(v) \cap \ZZ^d}z^x \cdot \prod_{j=1}^d (1-z_j)& \leq   \lim_{z \to (1,\ldots,1)^-} \sum_{i=1}^{d} \sum_{x \in (R_i(v) \cap \ZZ^d)}kz^x \cdot \prod_{j=1}^d (1-z_j) \notag \\
  &= \lim_{z \to (1,\ldots,1)^-}  \sum_{i=1}^{d} k (1-z_i^{{\alpha(v)}_i-\beta(v)_i}) z^{{\beta(v)}}=0.
  \end{align}

Combining
Equation \ref{equation stripes} and Equation \ref{equation gamma}, we get:
\begin{equation}\label{right hand side}
\lim_{z \to (1,\ldots,1)^-}  \sum_{x \in \ZZ^d_{\geq {\alpha(v)}}  } kz^x \cdot \prod_{j=1}^d (1-z_j) + \sum_{x \in \Gamma(v) \cap \ZZ^d}z^x \cdot  \prod_{j=1}^d (1-z_j) =k,
\end{equation} 
which shows that the left side of Equation \ref{equation no limit} converges to $k$ as $z$ converges to $(1,\ldots,1)$ from below.
Note that
if   $ \lim_{z \to (1,\ldots,1)^-} s_i(z)$  exists for all $i$,  then 
taking the limit of Equation \ref{equation no limit} as $z$ 
 converges to $(1,\ldots,1)$ from below will give us
Equation \ref{equation in props asymptotic}, and the proof will be done.
However, the limit of $s_i(z)$ does not always exist, and  hence we need a more subtle approach to derive Equation \ref{equation in props asymptotic}.


Since $|z_i|<1$ for all $i$,  we have  that $s_i(z)$ is a positive real number for all $i$.
Also note that 
the multiplicity of every element in $\Lambda_i$ can not exceed $k$,
and every element of $\Lambda_i-a_i$ is contained in $\ZZ^d_{\geq -a_i}$.
These facts allow us to derive the following inequality
\begin{align}\label{equation s is bounded above}
s_i(z)= \sum_{y\in \Lambda_i -a_i} z^y \cdot \prod_{j=1}^d (1-z_j)
&\leq \sum_{y \in  \ZZ^{d}_{\geq -a_i}} k z^y \cdot \prod_{j=1}^d (1-z_j) \notag \\
&= k z^{[-a_i]} \cdot \prod_{j=1}^{d} \frac{1}{1-z_j} \cdot \prod_{j=1}^d (1-z_j) \notag\\
&= k z^{[-a_i]},
\end{align} 
where $[x]$ is the integer part of the real vector $x$ in $\RR^d$. 
Note that as $z$ converges to $(1,\ldots,1)$ from below, the right side of Equation \ref{equation s is bounded above} is bounded from above by $k+1$.
Hence as $z$ converges to $(1,\ldots,1)$ from below, the value of $s_i(z)$  is bounded between $0$ and $k+1$. 
The Bolzano-Weierstrass theorem \cite{BS92} then implies that there exists a sequence $(\mathfrak z_u)_{u \in \NN}$ that converges to $(1,\ldots,1)$ from below and with the property that $\lim_{u \to \infty} s_i(\mathfrak z_u)$ exists for all $i$.

 Let  $g_i:=\lim_{u \to \infty} s_i(\mathfrak z_u)$ for all $i$, note that $g_i$ is non-negative because  $s_i(\mathfrak z_u)$ is a  positive real number for all $u \in \NN$. 
 Also note that  the definition of $s_i(z)$  does not involve $v$, 
 and hence each $g_i$ is a constant that is independent from the choice of $v$. 
By substituting  $\mathfrak{z}_u$ into Equation \ref{equation no limit} and then taking the limit as $u$ goes to infinity,  we get the following equality:
\begin{align}
 & \sum_{i=1}^{n} \# (\ZZ^d \cap \{ P^{\h}+a_i+v \} ) \cdot \lim_{u \to \infty}s_i(\mathfrak z_u) \notag \\
= &    \lim_{u \to \infty} \sum_{x \in \ZZ^d_{\geq {\alpha(v)}}  } k \mathfrak z_u^x \cdot \prod_{j=1}^d (1- {\mathfrak z_u}_j) + \sum_{x \in \Gamma(v) \cap \ZZ^d} \mathfrak z_u^x \cdot  \prod_{j=1}^d (1-{\mathfrak z_u}_j) \notag.
\end{align}
Substituting  Equation \ref{right hand side} into the right side of the equation above, we get:
\begin{align}\label{equation last in the lemma}
\sum_{i=1}^{n}  \# (\ZZ^d \cap \{ P^{\h}+a_i+v \} ) \cdot g_i = & k. \end{align}
To get Equation \ref{equation in props asymptotic} from Equation \ref{equation last in the lemma}, note that  
\begin{equation}\label{equation technical}
\# (\ZZ^d  \cap \{ P^{\h}+a_i+v \} ) =\# (-1 \cdot \ZZ^d  \cap \{ -1 \cdot P^{\h} -a_i-v )\} )=L_{\ZZ^d }^{\h}(-a_i-v).
\end{equation}
 Substituting Equation \ref{equation technical} into the left side of Equation \ref{equation last in the lemma}, we get:
 \begin{align}\label{equation final in the lemma}
\sum_{i=1}^{n} L_{\Lambda}^{\h}(-a_i-v) \cdot g_i = & k. 
\end{align}
As the choice of $v \in \RR^d$ is arbitrary, we can replace $v$ in Equation \ref{equation final in the lemma} by $-v$ to get 
  Equation \ref{equation in props asymptotic} and the proof is now complete. 
\qed \end{proof}

\begin{remark}
The proof of Lemma \ref{asymptotic method} can be made much shorter if  we use the  stronger assumption that $\Lambda$ is equal to the disjoint union of finitely many translates of one lattice.
Indeed, in this case, we have:
\begin{align*}
k=\# (\Lambda \cap \{-1 \cdot P^{\h} + v\}) &= \sum_{i=1}^n \# ( 
\{a_i+ \lat
\} \cap \{-1 \cdot P^{\h} + v\}) \\&= \sum_{i=1}^n \# (\lat \cap \{  -1 \cdot P^{\h} + v-a_i\})=\sum_{i=1}^{n} L_{\lat}^{\h}(v-a_i), \notag 
\end{align*}
for all $v$ in $\RR^d$. 
\end{remark}

\begin{remark}
The proof of Lemma \ref{asymptotic method} no longer works if  we use the  weaker  assumption that every element of $\Lambda$ is contained in a quasi-periodic set, because   the original assumption is essential for  deriving Equation \ref{equation intermediate} from Equation \ref{equation start}.
\end{remark}


Now we show that
 the value $g_1,\ldots, g_n$ in Lemma \ref{asymptotic method} can, in fact, be chosen to be rational numbers with some value $m$ on the right side of Equation \ref{equation in props asymptotic}, which gives us the following theorem.
 
\begin{theorem}\label{main theorem 3} 
Let $P$ be a convex polytope that $k$-tiles $\RR^d$ with a discrete multiset $\Lambda$.
Let $\lat$ be a lattice and let $a_1,\ldots, a_n$ be vectors in $\RR^d$.
If    every element of $\Lambda$ is contained in the 
finite union  $\bigcup_{i=1}^{n} a_i+\lat$, 
 then $P$  $m$-tiles in $\RR^d$ for some $m$ with a finite union of copies of the lattices $a_1+\lat$, $, \ldots,$  $a_n+\lat$. 
\end{theorem}
\begin{proof}
 By  Lemma \ref{asymptotic method}, there are non-negative real numbers $g_1,\ldots, g_n$ such that
 \begin{equation}\label{equation vector space}
 \sum_{i=1}^{n}   g_i \cdot L_{\lat}^{\h}(v-a_i)=k,
\end{equation}
for  all vectors $v$ in $\RR^d$. 
For an arbitrary $v$ and $w$ in $\RR^d$, let $l_i(v,w)$ be the integer 
\[l_i(v,w):= L_{\lat}^{h}(v-a_i) -L_{\lat}^{h}(w-a_i),\]
and  let $V$  be the vector space   in $\RR^n$  spanned by the following set of vectors:
\[ \{ (l_1(v,w),l_2(v,w), \ldots, l_n(v,w) ) \ : \ v,w \in \RR^d \}.\]

Note that by Equation \ref{equation vector space}, the vector 
$ (g_1, \ldots, g_n)$ is contained in the orthogonal complement $V^\perp$ of $V$, and hence $V^\perp$ contains a non-zero non-negative vector.
Also note that    $V$ are generated by integer vectors, and hence  $V$ and $V^{\perp}$ have a basis of integer vectors. 
These two facts imply that 
   there is a non-negative non-zero integer vector $(g_1', \ldots g_n')$ that is contained in $V^{\perp}$.
By the construction of the vector space $V$,  the statement that $(g_1', \ldots g_n')$  is orthogonal to $V$ is equivalent to the following equation:
 \begin{equation}
 m=\sum_{i=1}^{n}   g_i' \cdot \#(\lat \cap \{-1\cdot P^{\h} +v-a_1\})= \sum_{i=1}^{n}   g_i' \cdot \#(a_i+\lat \cap \{-1\cdot P^{\h} +v\}), 
\end{equation}
for some positive integer $m$ and for all $v$ in $\RR^d$. By Lemma \ref{lemma GRS} this implies that $P$ $m$-tiles $\RR^d$ with the union of the translated lattices $a_1+\lat$, $, \ldots,$  $a_n+\lat$,  where each element of  $a_i+\lat$ has multiplicity  $g_i'$.  
\qed \end{proof}

In the case when all elements of $\Lambda$ are contained in a lattice, Theorem \ref{main theorem 3} gives us the following corollary:
\begin{repcorollary}{corollary one lattice} 
Let  $P$ be a convex polytope that $k$-tiles $\RR^d$ with a discrete multiset $\Lambda$.
 If every  element of $\Lambda$ is  contained in a lattice $\lat$, then $P$  $m$-tiles $\RR^d$ with $\lat$ for some $m$. \qed 
\end{repcorollary}



We now present a technical  lemma that allows us to reduce Theorem \ref{main theorem 1}  to the situation where all elements of $\Lambda$ are contained in a lattice, so that we can apply Corollary \ref{corollary one lattice} to prove Theorem \ref{main theorem 1}.

 \begin{lemma}\label{general lambda}
Let $P$ be a convex polytope that $k$-tiles $\RR^d$ with a discrete multiset $\Lambda$, and suppose that $\Lambda$ is the disjoint union of two  discrete multisets $\Lambda_1$ and $\Lambda_2$.
If the set $\RR^d \setminus (\partial P+ \Lambda_1) \cap (\partial P+\Lambda_2)$ is path-connected and $\Lambda_1$ is non-empty, then $P$  $m$-tiles $\RR^d$ with  $\Lambda_1$  for some $m$. 
\end{lemma}

\begin{proof}
By Lemma \ref{lemma GRS}, the fact that $P$ $k$-tiles $\RR^d$ with $\Lambda=\Lambda_1 \sqcup \Lambda_2$ implies that
\begin{equation}\label{equation general lambda}
L_{\Lambda_1}^{\h}(v)+ L_{\Lambda_2}^{\h}(v)= L_{\Lambda}^{\h}(v)=k,
\end{equation}
for all $v$ in $\RR^d$.

Let  $v_1$ and $v_2$ be two points in $\RR^d \setminus (\partial P+ \Lambda_1) \cap (\partial P+\Lambda_2)$. 
Because $\RR^d \setminus (\partial P+ \Lambda_1) \cap (\partial P+\Lambda_2)$ is path-connected, there is a path $\mathfrak P: [0,1] \to \RR^d$  starting at $v_1$ and ending at $v_2$  such that $\mathfrak P$ does not contain  points from $(\partial P+ \Lambda_1) \cap (\partial P+\Lambda_2)$. 
We claim that the function $L_{\Lambda_1}^{\h}(\mathfrak P(x))$ remains constant $x$ goes from $0$ to $1$.
 
Suppose to the contrary that  $L_{\Lambda_1}^{\h}(\mathfrak P(x))$ is not a constant function.
 This means that there is   $\alpha \in [0,1]$ such that the function $L_{\Lambda_1}^{\h}(\mathfrak P(x))$ is not  constant  in every open neighborhood of $\alpha$. 
 Because $\mathfrak P (\alpha)$ is not contained in $(\partial P+ \Lambda_1) \cap (\partial P+ \Lambda_2)$,  either one of the following scenarios  will hold:
\begin{itemize}[leftmargin=*]
\item $\mathfrak P (\alpha) $ is not contained in $\partial P+\Lambda_1$. This means that $\mathfrak P (\alpha)$ is in  general position  with respect to $\Lambda_1$. By  Property \ref{property general position}  in Section \ref{section lattice point enumeration}, the function $L_{\Lambda_1}^{\h}$  is constant in a sufficiently small neighborhood of $\mathfrak P(\alpha)$, contradicting the assumption on  $\alpha$.

 \item $\mathfrak P (\alpha) $ is not contained in $\partial P+\Lambda_2$. This means that $\mathfrak P (\alpha)$ is in  general position  with respect  to $\Lambda_2$.
  By  Property \ref{property general position}   in Section \ref{section lattice point enumeration}, the function $L_{\Lambda_2}^{\h}$  is constant in a sufficiently small neighborhood of $\mathfrak P(\alpha)$.
By Equation \ref{equation general lambda} we have   
$L_{\Lambda_1}^{\h}= k- L_{\Lambda_2}^{\h}$, and hence $L_{\Lambda_1}^{\h}$ is also a constant function in a sufficiently small neighborhood of $\mathfrak P(\alpha)$,  contradicting to the assumption on $\alpha$.
  \end{itemize}
Hence    $L_{\Lambda_1}^{\h}(v)$ has a constant value $m$ for all $v$ in $\RR^d\setminus (\partial P+ \Lambda_1) \cap (\partial P+ \Lambda_2)$.

We now show that $m$ is a positive integer. 
Because $\Lambda_1$ is non-empty, we have   $L_{\Lambda}^{\h}(v_1)= \# (\Lambda \cap  \{-1 \cdot P^{\h}+ v_1 \})$ is positive   for some $v_1$ in an open set $B$ of $\RR^d$.
Because the set 
 $ \RR^d \setminus (\partial P+ \Lambda_1) \cap (\partial P +\Lambda_2)$  is  dense in $\RR^d$, there exists $v_2 \in  \RR^d \setminus (\partial P+ \Lambda_1) \cap (\partial P +\Lambda_2)$ that is also contained in $B$.
This implies that $m=L_{\Lambda_1}^{\h}(v_2)=L_{\Lambda_1}^{\h}(v_1)>0$, and hence $m$ is a positive integer.


Now let $v \in \RR^d$ be a vector in  general position with respect to $\Lambda_1$.
By Property \ref{property general position} in Section \ref{section lattice point enumeration},   
the function $L_{\Lambda_1}^{\h}(v)$ is a constant function in an open neighborhood $B_v$ of $v$ in $\RR^d$.
Because the set 
 $ \RR^d \setminus (\partial P+ \Lambda_1) \cap (\partial P +\Lambda_2)$  is  dense in $\RR^d$,  there exists $w$ in $\RR^d \setminus (\partial P+ \Lambda_1) \cap (\partial P +\Lambda_2)$ that is  contained  in $B_v$.
 By the argument above, this implies that $L_{\Lambda_1}^{\h}(v)= L_{\Lambda_1}^{\h}(w)=m$.
 Because the choice of $v$ is arbitrary, this implies that $L_{\Lambda_1}^{\h}(v)=m$ for every  $v \in \RR^d$ that is in  general position with respect to $\Lambda_1$. 
 By Lemma \ref{lemma GRS}, we conclude that $P$ $m$-tiles $\RR^d$ with $\Lambda_1$.
\qed \end{proof}

We  now proceed  to prove Theorem \ref{main theorem 1}. 
\begin{reptheorem}{main theorem 1}
Let $P$ be a convex polytope that $k$-tiles $\RR^d$ with a discrete multiset $\Lambda$, and suppose that every element  of $\Lambda$ is  contained in a quasi-periodic set $\Qpl$. If a lattice $\lat$ in $\Qpl$ is in  general position with respect to $\Qpl$ and $\lat \cap \Lambda$ is non-empty,  then $P$ $m$-tiles $\RR^d$ with $\lat$ for some $m$. 
\end{reptheorem}

\begin{proof}
Without loss of generality, we can assume that $\lat_i$ is a lattice instead of a translate of a lattice. 
 Let $\Lambda_1=\Lambda \cap \lat_i$ and $\Lambda_2= \Lambda \setminus \Lambda_1$. 
 We have  that $\Lambda$ is a disjoint union of $\Lambda_1$ and $\Lambda_2$, and by the distributive law  the set $H=(\partial P+\Lambda_1) \cap \partial (P+\Lambda_2)$ is contained in $H_i$ (where $H_i$ is as defined in Equation \ref{technical condition}). 
 Because $\RR^d \setminus H_i$ is path-connected by the assumption that $\lat$ is in general position with respect to $\Qpl$, this implies  that $\RR^d \setminus H$ is path-connected.
 Also note that by assumption $\Lambda_1$ is a non-empty multiset.
 Hence by Lemma \ref{general lambda} $P$ $m'$-tiles $\RR^d$ with $\Lambda_1$ for some $m'$.
 Because every element of $\Lambda_1$ is contained in $\lat_i$, we conclude that $P$  $m$-tiles $\RR^d$ with $\lat_i$ for some $m$ by Corollary \ref{corollary one lattice} . 
\qed  \end{proof}

Note  that the assumption that the lattice in Theorem \ref{main theorem 1}  is in  general position  can not be omitted from the statement of the theorem, as  seen in Example \ref{example rectangle} below.

\begin{example}\label{example rectangle}
Let $P$ be a rectangle in $\RR^2$ with $(0,0) , (0,\frac{1}{2}), (1,0), (1, \frac{1}{2})$ as vertices.
Let $\lat_1$ be the lattice $\ZZ^2$, let $\lat_2$ be the translated lattice $( \frac{\sqrt 2}{2},\frac{1}{2})+\ZZ^2$, and let $\Qpl=\lat_1 \cup \lat_2$.
It can be seen from Figure \ref{figure tiling} that $P$ $1$-tiles $\RR^2$ with $\Qpl$, but
$P$ does not $m$-tile $\RR^2$ with $\lat_1$ or $\lat_2$ for any $m$.  
 Also notice that $\lat_1$ and $\lat_2$ are not in  general position to $\Qpl$, as the sets
 $\RR^2 \setminus H_1=\RR^2 \setminus H_2= \RR^2 \setminus \{(x,y) \in \RR^2 \ : \ x \in \ZZ \}$ are not path-connected.
 \end{example}

  \begin{figure}[ht]\label{figure tiling}
  \caption{A rectangle $P$ that 1-tiles $\RR^2$ with $\Qpl=\lat_1 \cup \lat_2$, but does not $m$-tile with $\lat_1$ or $\lat_2$.}
  \centering
    \includegraphics{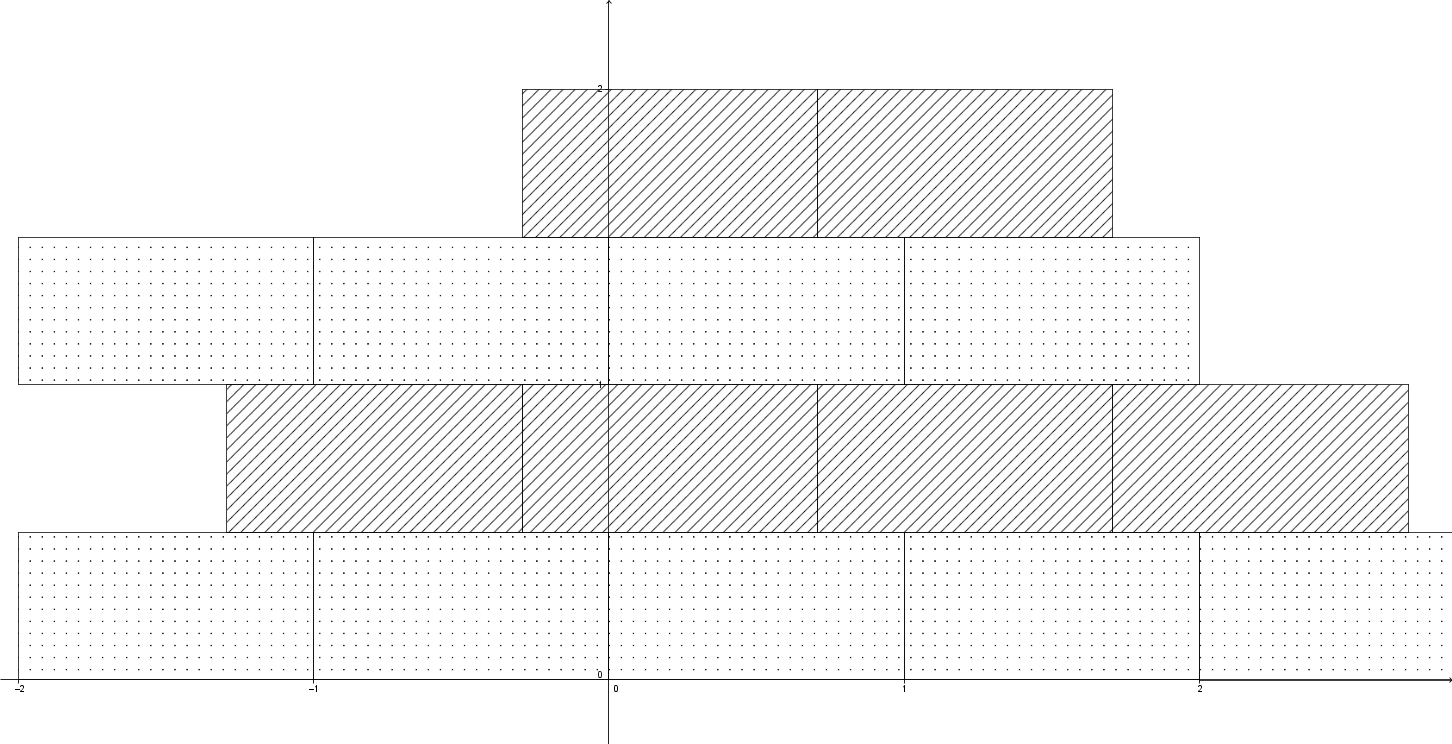}
\end{figure}

\section{Quasi-periodic tiling without Hypothesis \ref{technical condition}} \label{section for main theorem 2}
In this section, we  discuss quasi-periodic tiling in the situation when  Hypothesis \ref{technical condition}  is omitted.
As we observed from Example \ref{example rectangle}, the condition that a lattice $\lat$ is in  general position with respect to the quasi-periodic set $\Qpl$ can not be dropped.
However, this does not
preclude the possibility that $P$ $m$-tiles $\RR^d$ for some $m$ with some lattice $\lat$ that is not contained in $\Qpl$.  
For example, the rectangle $P$ in Example \ref{example rectangle}
can $2$-tile $\RR^2$ with the lattice $\frac{1}{2}\ZZ\times \frac{1}{2}\ZZ$, even though $\frac{1}{2}\ZZ\times \frac{1}{2}\ZZ$ is not contained in $\Qpl$.
In the next theorem we present an approach to construct such a lattice $\lat$ for the case when $\Qpl$ is a union of two translated copies of a lattice.
\begin{reptheorem}{main theorem 2}
Let $P$ be a  convex polytope that $k$-tiles $\RR^d$ with a discrete multiset $\Lambda$. 
Let $\lat_1$ and $\lat_2$ be  translates of one single lattice in $\RR^d$.
If every element in $\Lambda$ is contained in $\lat_1 \cup \lat_2$, then  $P$  $m$-tiles $\RR^d$ with some lattice $\lat$ for some $m$. 
\end{reptheorem}
 \begin{proof}
Without loss of generality, we can assume that $\lat_1=\ZZ^d$ and $\lat_2=a+\ZZ^d$ for some $a$ in $\RR^d$. 
Suppose that  $a$ is a rational vector, let $N$
  be the least common multiple of the denominator of entries of $a$.
Note that  both $\ZZ^d$ and $a+\ZZ^d$ are now contained   
  $(\frac{1}{N}\ZZ)^d$, and by Corollary \ref{corollary one lattice} we have  $P$ $m$-tiles $\RR^d$ with $(\frac{1}{N}\ZZ)^d$ for some $m$. 
Hence we can assume that $a$ is not a rational vector.


By permuting the coordinates, we can  assume that $a=(\alpha_1, \alpha_2, \ldots, \alpha_k, \beta_{k+1}, \ldots, \beta_d)$, where $\alpha_1, \ldots, \alpha_k$ are irrational numbers linearly independent over $\QQ$, and $\beta_{k+1}, \ldots, \beta_d$ are contained in $\langle \alpha_1, \ldots, \alpha_k,1 \rangle_{\QQ}$. 
 Because $a$ is not a rational vector, we have $k\geq 1$.  
 Because $\beta_i$ is contained in $\langle \alpha_1, \ldots, \alpha_k,1 \rangle_{\QQ}$
for all $i \in \{k+1,\ldots, d\}$, 
there exists  $c_{i,j} \in \QQ $ for ${j \in \{1,\ldots,k\}}$ such that    $\beta_i=c_{i,1}\alpha_1 +\ldots+ c_{i,k} \alpha_k+c_{i,k+1}$. 
Let $N$ be the least common multiple of the denominators of $c_{i,j}$, where $i \in \{k+1,\ldots,d\}$ and $j \in \{1,\ldots,k\}$.
Note that  $\beta_i$ is contained in $\langle \alpha_1, \ldots, \alpha_k,1 \rangle_{(\frac{1}{N}\ZZ)^d}$ for all $i \in \{k+1,\ldots,d\}$. 

In the rest of this proof, we will use $L^{\h} $ as a shorthand for $L_{(\frac{1}{N}\ZZ)^d}^{\h} $.
By Lemma \ref{asymptotic method}, we have the following equation:
\begin{equation}\label{equation cute}
g_1 \cdot L^{\h}(v)+  g_2 \cdot L^{\h}(v-a)=k,\end{equation}
for some non-negative real numbers $g_1$ and $g_2$ and for all $v$ in $\RR^d$. 
If $g_1=0$, then $ L^{\h}(v-a)=\frac{k}{g_2}$ for all $v \in \RR^d$.
By Lemma \ref{lemma GRS}, this implies that $P$ $m$-tiles $\RR^d$ with $(\frac{1}{N}\ZZ)^d$ for $m=\frac{k}{g_2}$ and the claim is proved.
By symmetry we get the same conclusion for when $g_2=0$.
Hence we can assume that both $g_1$ and $g_2$ are non-zero.

Let  $L_j=L^{\h}(v-a\cdot j)-\frac{k}{g_1+g_2}$ for all  $j \in \ZZ$.
 Substituting $v$ in Equation \ref{equation cute} with $v-a\cdot j$, we get the following relation for all  $j \in \ZZ$:
\begin{equation}\label{equation L_j and L_j+1}
L_jg_1+L_{j+1}g_2=0,
\end{equation}
and we can without  loss of generality assume that $g_2 \leq g_1$.

 First, suppose that $g_2<g_1$.
   Note that by Equation \ref{equation L_j and L_j+1}  we have $L_{0}=(-\frac{g_2}{g_1})^jL_j$ for all  $j \in \ZZ$.
On the other hand, the function $L^\h$ is a periodic function by Property  \ref{property periodic} in Section \ref{section lattice point enumeration}, which implies that $L_j(v)$ (which is equal to $L^{\h}(v-a)-\frac{k}{g_1+g_2}$) is a bounded function.
 Because $L_j$ is bounded and $g_2<g_1$, we have 
 \[L_0=\lim_{j \to \infty} \left(-\frac{g_2}{g_1} \right)^jL_j= 0.\]
This implies that 
     $ L^{\h}(v)=\frac{k}{g_1+g_2}$ for all $v$ in $\RR^d$, and by Lemma \ref{lemma GRS} we have that $P$ $m$-tiles $\RR^d$ with $(\frac{1}{N}\ZZ)^d$ for $m= \frac{k}{g_1+g_2} $, and the claim is proved.

Now suppose that $g_1=g_2$.  
 We claim that
 $L_{2j+1}=L_0$ for some   $j \in \ZZ$.
 Note that if the claim holds,  then  we can conclude that $L_0=L_1$ (because  $L_{2j+1}=L_1$ and $L_{2j}=L_0$  for all $j \in \ZZ$ by Equation \ref{equation L_j and L_j+1}).
 Because we also have $L_0+L_1=0$ by Equation \ref{equation L_j and L_j+1}, this implies that $L^{\h}(v)-\frac{k}{g_1+g_2}=L_0=0$.
 By Lemma \ref{lemma GRS}, we can then conclude that $P$ $m$-tiles $\RR^d$ with $(\frac{1}{N}\ZZ)^d$ for $m= \frac{k}{g_1+g_2} $, and the claim is proved.




For any two points $w$ and $w'$ in $\RR^d$ and a constant $\epsilon>0$, we say that $w$ is $\epsilon$-close to $w'$ modulo $\ZZ^d$ if $|w-w'+\lambda|<\epsilon$ for some $\lambda$ in $\ZZ^d$. 
Let $a' \in \RR^k $ be the vector $(\alpha_1, \ldots, \alpha_k)$, 
where $\alpha_1,\ldots, \alpha_k$ are irrational numbers defined in the beginning of the proof.
We claim that for any $\epsilon>0$, there exists  $j\in \ZZ$ such that $(2j+1)\cdot a'$ is  $\epsilon$-close to  $0$ modulo $\ZZ^k$.
To prove this claim,
 we  use a  powerful tool from number theory called  the Weyl criterion for the multidimensional case. 
 
 We say that a sequence $(x_n )_{n\in \NN} $ of vectors in  $\RR^k$ is dense   modulo $\ZZ^k$  in $\RR^k$ if for any point $w$ in $\RR^k$ and a constant $\epsilon>0$, there exists a natural number $j$ such that $w$ is $\epsilon$-close to $x_j$ modulo   $\ZZ^k$. 

\begin{theorem} \label{Weyl Criterion} \textnormal{(weak form of Weyl criterion, \cite[Theorem 6.2]{Kuipers})}
Let $(x_n )_{n\in \NN} $ be a sequence of vectors in $\RR^k$. 
The sequence $(x_n )_{n\in \NN} $  is dense   modulo $\ZZ^k$  in $\RR^k$ if  for every lattice point $h\in \ZZ^k$, $h \neq 0$,
\[\lim_{M \to \infty} \frac{1}{M} \sum_{n=1}^{M}e^{2\pi i \langle h,x_n \rangle}=0. \] \qed 
\end{theorem}
Let $(x_n )_{n\in \NN} $ be the sequence defined by $x_n=2n \cdot a'$ for all $n \in \NN$. 
Because $\alpha_1, \ldots, \alpha_k$ are irrational numbers that are linearly independent over $\QQ$, we have that  ${ \langle h,a'\rangle}$ is not equal to $0$ for all  $h \in \ZZ^k$. 
Hence the limit of the sum in the Weyl criterion  is equal to 
\[ \lim_{M \to \infty}  \frac{1}{M} \sum_{n=1}^{M}e^{2\pi i \langle h,x_n \rangle}=  \lim_{M \to \infty}  \frac{1}{M} \sum_{n=1}^{M}e^{4n  \pi i \langle h,a' \rangle}= \lim_{M \to \infty}  \frac{1}{M} \cdot \frac{e^{4\pi i \langle h,a' \rangle}(1- e^{4M\pi i \langle h,a' \rangle})}{1-e^{4\pi i \langle h,a' \rangle}} = 0.\]
Hence the Weyl criterion implies that the sequence $(2n \cdot a')_{n \in \NN}$ is dense  modulo $\ZZ^k$  in $\RR^k$.
In the particular case when $w=-a'$, we have that  for any $\epsilon>0$, there exists  $j \in \NN$ such that $-a'$  is $\epsilon$-close to  $2j\cdot a'$ modulo $\ZZ^k$.
Hence  we conclude that $(2j+1) \cdot a'$ is $\epsilon$-close to $0$ modulo $\ZZ^k$. 

Because
$(2j+1) \cdot a'$ is $\epsilon$-close to $0$ modulo $\ZZ^k$, 
this implies that   
 $(2j+1)\alpha_1, \ldots, (2j+1)\alpha_k$ are all $\epsilon$-close to an integer. Because $\beta_i$ is contained in $\langle \alpha_1, \ldots, \alpha_k,1 \rangle_{(\frac{1}{N}\ZZ)^d}$, this also implies that $(2j+1)\beta_i$ is $O(\epsilon)$-close to $\frac{1}{N} \ZZ$ for $i \in \{k+1,\ldots, d\}$. 
 Hence we conclude that we can find an odd number $2j+1$ such that $(2j+1)\cdot a$ is $O(\epsilon)$-close to $(\frac{1}{N}\ZZ)^d$. 
 
 We will now show that $L_{2j+1}=L_0$. 
 Because we assume that $v \in \RR^d$ is in  general position with respect to $(\frac{1}{N} \ZZ)^d$,  by Property \ref{property general position} Section \ref{section lattice point enumeration} there is a sufficiently small open neighborhood $B_v$ of $v$ such that $L^{\h}$ is a constant function in $B_v$. 
 By our previous argument, for any $\epsilon>0$, there is an odd number $2j+1$ such that $(2j+1) \cdot a$ is $O(\epsilon)$-close to $(\frac{1}{N}\ZZ)^d$.  
 By choosing a sufficiently small $\epsilon$, we conclude that $v-(2j+1)\cdot a$ is contained in $B_v$ modulo $(\frac{1}{N}\ZZ)^d$.
 Because the function  $L^{\h}$ has period $(\frac{1}{N}\ZZ)^d$ (Property \ref{property periodic} Section \ref{section lattice point enumeration}),
 this implies that $L_{2j+1}=L^{\h}(v-(2j+1)\cdot a) = L^{\h}(v)=L_0$, and the proof is complete. 
\qed \end{proof}

\begin{remark}
 The proof of Theorem \ref{main theorem 2} can not be altered in any way  to show that $P$  $m$-tiles $\RR^d$  with $\ZZ^d$ instead of $(\frac{1}{N}\ZZ)^d$. 
This can be seen from  Example \ref{example rectangle}, where the rectangle in  the example does not $m$-tile $\RR^2$ with $\ZZ^2$ for any $m$.
\end{remark}

\section{Future research}\label{conjectures}
We conclude this paper by discussing possible future research problems that may lead to a proof of Conjecture \ref{conjecture last boss}.

 \begin{problem} \label{conjecture theorem 1.2}
  Let  $P$ be a convex polytope that $k$-tiles $\RR^d$ with a discrete multiset $\Lambda$, and suppose that every element of $\Lambda$ is  contained in a  quasi-periodic set $\Qpl$. Prove or disprove that $P$ $m$-tiles $\RR^d$ with a lattice $\lat$ for some $m$.
 \end{problem} 
Problem \ref{conjecture theorem 1.2} is  a generalization of Theorem \ref{main theorem 1} by removing Hypothesis \ref{technical condition}.
A more specific question to ask is whether Problem \ref{conjecture theorem 1.2} has a positive answer when $\Qpl$ is a finite union of translates of a single  lattice.
 A positive answer to this specific problem is given by Theorem \ref{main theorem 2}   in the case where $\Qpl$ is a union of two translates of a single  lattice.


\begin{problem}\label{conjecture bridge}
Prove or disprove that if a convex polytope $P$ $k$-tiles $\RR^d$, then $P$  $m$-tiles $\RR^d$ with a discrete multiset $\Lambda$ that is contained in a quasi-periodic set $\Qpl$  for some $m$.
\end{problem} 
For dimension 2 and 3, Problem \ref{conjecture bridge} was positively answered by \cite{Kolountzakis} and \cite{GKRS} respectively.
This problem is open for dimensions higher than 3.

In particular, 
a positive answer to both Problem \ref{conjecture theorem 1.2} and Problem \ref{conjecture bridge}  will imply that Conjecture \ref{conjecture last boss} is true.

\section*{Acknowledgement}
The author would like to thank Sinai Robins for his advice and helpful discussions during the preparation of this paper and for introducing the author to this topic.
 The author would like to thank the anonymous reviewers
for their valuable comments and suggestions to improve the
quality of the paper. 
Lastly, the author would like to thank  Henk Hollmann and Thomas Gavin for proofreading this paper.

%
%

\bibliographystyle{alpha}
\bibliography{quasi}

\end{document}